\documentclass[12pt,a4paper]{amsart} 
\usepackage[dvips]{epsfig} 
\usepackage{amsmath,amssymb,euscript,graphicx,amsfonts,verbatim, bm}
\usepackage{enumerate,color}
\headsep=1cm \numberwithin{equation}{section}
\newtheorem{theorem}{Theorem}[section]
\newtheorem{corollary}[theorem]{Corollary}
\newtheorem{lemma}[theorem]{Lemma}

\theoremstyle{definition}

\newtheorem{remark}[theorem]{Remark}

\numberwithin{equation}{section}
\usepackage[linktoc=all, hyperindex]{hyperref}% 
\hypersetup{colorlinks, citecolor=green, filecolor=blue, linkcolor=red, urlcolor=black}
\newcommand\Spec{\operatorname{Spec}}

\newcommand{\im}{\operatorname{Im}}

\newcommand{\charac}{\operatorname{char}}

\def\NN{{\mathbb N}}
\def\ZZ{{\mathbb Z}}

\begin{document}
\title[exceptional units in quaternion rings]{On the sumsets of exceptional units in quaternion rings}
%\author[Cheraghpour \and Dol{\v z}an]{Hassan Cheraghpour \and David Dol{\v z}an}

\author{Hassan Cheraghpour$^{*}$ and David Dol\v{z}an}

\address{Hassan Cheraghpour:
University of Primorska, FAMNIT, Glagolja{\v s}ka 8, 6000 Koper, Slovenia.}
\email{cheraghpour.hassan@yahoo.com}

\address{David Dol{\v z}an:
$ ^{1} $Department of Mathematics, Faculty of Mathematics and Physics, University of Ljubljana, Jadranska 21, 1000 Ljubljana, Slovenia \and $ ^{2} $IMFM, Jadranska 19, 1000 Ljubljana, Slovenia.}
\email{david.dolzan@fmf.uni-lj.si}

%\subjclass{}%
\thanks{The first author is supported in part by the Slovenian Research Agency (research project N1-0210).}
\thanks{The second author is supported in part by the Slovenian Research Agency (research core funding No. P1-0222).}
\thanks{$^{*}$Corresponding author: Hassan Cheraghpour}
%\date{}%
%\dedicatory{}%
%\commby{}%
% --------------------------------------------------------------------------------------
\maketitle
% 1-----------------------------------------------------------------------------------------------

\begin{abstract}
\noindent
We investigate sums of exceptional units in a quaternion ring $H(R)$ over a finite commutative ring $R$. We prove that in order to find the number of representations of an element in $H(R)$ as a sum of $k$ exceptional units for some integer $k \geq 2$, we can limit ourselves to studying the quaternion rings over local rings. For a local ring $R$ of even order, we find the number of representations of an element of $H(R)$ as a sum of $k$ exceptional units for any integer $k \geq 2$. For a local ring $R$ of odd order, we find either the number or the bounds for the number of representations of an element of $H(R)$ as a sum of $2$ exceptional units.

\end{abstract}

\maketitle

\section{\textbf{Introduction}} 
\noindent
%An element $ u $ of $ R $ is called a \textit{unit} if there exists an element $ v \in R $ such that $ uv=1_{R}=vu $. 
Much research has been devoted to studying sums of units in a ring in recent times. For example, it has been proved in \cite{lanski} that every element of an Artinian ring can be expressed as the sum of two units if and only if the factor ring over the Jacobson radical does not contain a summand isomorphic to the field with two elements.  

Furthermore, in 1969 Nagell \cite{nagell} introduced the concept of an exceptional
unit - the unit $u \in R$ is called \emph{exceptional} if the element $1 - u \in R$ is also a unit in $R$.
It turns out that exceptional units are important for studying Diophantine equations, since the solutions of many Diophantine equations can be reduced to the solution of the equation $ax + by = 1$, where $x$ and $y$ are units, so one has to find the exceptional units in the case when $a = b = 1$. The importance of the exceptional units also surfaces when studying certain cubic Diophantine equations \cite{nagell}, Thue equations \cite{tzanakis1}, Thue-Mahler equations \cite{tzanakis2} and discriminant form equations \cite{smart}.
Furthermore, exceptional units also have connections with cyclic resultants \cite{stewart1, stewart2}, Euclidean number fields \cite{houriet, lenstra, leutbecher} and Lehmer’s conjecture related to Mahler’s measure \cite{silverman1, silverman2}.

The sums of exceptional units in a ring have also been studied recently. Let $\varphi_k(R,c)$ denote the number of representations of $c \in R$ as a sum of $k$ exceptional units in a ring $R$. 
In \cite{yang}, the authors proved a formula for $\varphi_k({\mathbb Z}_n,c)$ where ${\mathbb Z}_n$ denotes the ring of integers modulo $n$, Miguel \cite{Mig18} found the number $\varphi_k(R,c)$ for an arbitrary finite commutative ring $R$, while in \cite{Dol19} the number $\varphi_k(R,c)$ has been studied in the setting of noncommutative rings. 

Set 
\begin{equation*}
H(R)= \{r_1+ r_2i + r_3j + r_4k : r_i \in R\}=R \oplus Ri \oplus Rj \oplus Rk,
\end{equation*}
where $ i^2 = j^2 = k^2 = ijk = -1$, and $ij = -ji$. Then, with the componentwise addition, multiplication subject to the given relations, and the convention that $ i,j,k $ commute with $ R $ elementwise, $H(R)$ is a ring called the $\it{quaternion}$ $\it{ring}$ over $R$, which is a generalization of the Hamilton's division ring of real quaternions $\mathbb{H}=H(\mathbb{R})$. Quaternion rings and their properties have also been quite heavily studied recently. For example, in \cite{arist1, arist2} the authors investigated the ring $H(\ZZ_p)$ for a prime number $p$, while in \cite{mig}, the number of idempotents and the number of zero-divisors in $H(\ZZ_p)$ were found.
In \cite{ghara}, the structure of superderivations of the quaternion rings over some special $\ZZ_2$-graded rings was determined, while in \cite{ghara2} the authors described the form of some mappings on quaternion rings.
Finally, in \cite{Cher22}, the structure of the ring $H(R)$ was described.

In this paper, we investigate the sums of exceptional units in a quaternion ring over a finite commutative ring. In the next section, we gather the definitions and known results that we shall use throughout the paper. We prove that in order to find the number of representations of an element in $H(R)$ as a sum of $k$ exceptional units, we can limit ourselves to studying the quaternion rings over local rings (see Theorem \ref{local}). Thereby, In Section 3, we study the quaternion rings over finite commutative local rings of even orders. We find the number of representations of any element as a sum of $k$ exceptional units for any integer $k \geq 2$ (see Theorem \ref{main}). In Section 4, we turn our attention to the quaternion rings over finite commutative local rings of odd orders. It turns out that the problem is much more complex in this case. However, we do manage to find the number of representations of many elements as a sum of $2$ exceptional units, and we find the bounds for this number in the remaining cases (see Theorem \ref{main2}). We manage to achieve this by utilizing the isomorphism between the quaternion ring and the ring of $2$-by-$2$ matrices over the factor field, and studying the properties of sumsets of exceptional units in the matrix case.

\bigskip

\section{\textbf{Preliminaries}}
\noindent

Let $R$ be a ring (with identity). We shall denote its Jacobson radical by $ J=J(R) $. A ring $ R $ is called \textit{local} if $ R/J $ is a division ring.
We denote the multiplicative group of units of $ R $ by $ R^{*} $ and we let $ |A| $ stand for the order
of any finite set $ A $. 
%A unit $ u \in R^{*} $ is called \textit{exceptional} if $ 1_{R}-u \in R^{*} $.
We write $ R^{**} $ for the set of all exceptional units of $ R $. For any integer $k \geq 2$ and any $c \in R $, we define $ \varphi_{k}(R, c) $ to be the number of representations of $ c $ as the sum of $ k $ exceptional units of R, i.e.,
\begin{equation*}
\varphi_{k}(R, c) = | \{(x_{1}, \dots , x_{k}) \in (R^{**})^{k} : x_{1} + \dots + x_{k} = c \}|.
\end{equation*}

We will denote the ring of $2$-by-$2$ matrices over a field $F$ by $M_2(F)$ and the group of invertible matrices therein by $GL_2(F)$. The identity matrix in $M_2(F)$ will be denoted by $I$, the spectrum of a matrix $A$ will be denoted by $\Spec(A)$ and the centralizer of an invertible matrix $A$ will be denoted by $C_{GL_2(F)}(A)$. For a matrix $A \in M_2(F)$ we will denote its entry at position $(i,j)$ by $A_{ij}$.
Finally, the finite field of order $q$ will be denoted by $GF(q)$ and the characteristic of a field $F$ will be denoted by $\charac F$.

In \cite{Cher19}, the following theorem has been stated.
\begin{theorem}\cite[Theorem 3.10]{Cher19}
\label{Th1}
Let $ F $ be a \emph{(}not necessarily finite\emph{)} field whose characteristic is an odd prime number and let $ R $ be an algebra over $ F $. Then
\begin{equation*}
H(R) \cong M_{2}(R).
\end{equation*}
\end{theorem}

Also, in \cite{Cher22}, Cheraghpour et al. proved the following theorem:
\begin{theorem}\cite[Theorem 3.6]{Cher22}
\label{Th3}
Let $ R $ be a ring with $ 2^{-1} \in R $. Then
\begin{equation*}
J(H(R))=H(J(R)).
\end{equation*}
\end{theorem}

Dol{\v z}an, in \cite{Dol19}, proved the following:
\begin{corollary}\cite[Corollary 2.3]{Dol19}
\label{Th2}
Let $ R $ be a finite ring with Jacobson radical $ J(R) $. Choose $ c \in R $ and $ k \geq 2$. Then 
\begin{equation*}
\varphi_{k}(R, c)=|J(R)|^{k-1}\varphi_{k}(R/J(R), \overline{c}),
\end{equation*}
where $\overline{c}=c+J(R)$.
\end{corollary}

The following lemma states that we may only consider the sums of exceptional units in directly indecomposable rings.

\begin{lemma}\cite[Lemma 2]{Mig18}
\label{lm4}
Let $ R \cong R_{1} \oplus \dots \oplus R_{t}$ be the decomposition of ring $ R $ into a direct sum of local rings. Let $ \psi: R \rightarrow R_{1} \oplus \dots \oplus R_{t} $ be an isomorphism. If $ c \in R $ and $ \psi(c)=(c_{1}, \dots, c_{t}) $, then
\begin{equation*}
\varphi_{k}(R,c)=\prod_{i=1}^{t}\varphi_{k}(R_{i},c_{i}).
\end{equation*}
\end{lemma}

Since any finite commutative ring can be decomposed into a finite direct sum of finite commutative local rings, we have the following result. 

\begin{theorem}
\label{local}
     Let $k \geq 2$ be an integer and let $R$ be a finite commutative ring, where $ R \cong R_{1} \oplus \dots \oplus R_{t} $ is the decomposition of $ R $ into finite local commutative rings $ R_{1}, \dots, R_{t} $.
   Then for any $c \in H(R)$ we have
   \begin{equation*}
      \varphi_{k}(H(R), c)=\prod_{i=1}^{t}\varphi_{k}(H(R_{i}), c_{i}),
   \end{equation*}
   where $ \psi : H(R) \rightarrow H(R_{1}) \oplus \dots \oplus H(R_{t}) $ defined by $ \psi(c)=(c_{1}, \dots, c_{t}) $ is a ring isomorphism.
\end{theorem}
\begin{proof}
Since $ R \cong R_{1} \oplus \dots \oplus R_{t} $, then it can be easily seen that $ H(R) \cong H(R_{1}) \oplus \dots \oplus H(R_{t}) $. The rest follows from Lemma \ref{lm4}.
\end{proof}

This now implies that we can limit ourselves to studying finite local commutative rings. In the next section, we shall investigate the local rings of even order, while the final section deals with the local rings of odd order.

\bigskip

\section{\textbf{Rings of even order}}
\noindent

It is well known that a finite local ring has order $p^t$ for some prime $p$ and integer $t$.
In this section, we examine the case when $R$ is a finite commutative local ring of order of a power of $2$. We have the following lemma, which generalizes Theorem 4.4 from \cite{Cher22}.

\begin{lemma}
\label{case2local}
   Let $R$ be a finite commutative local ring such that $2$ is a zero-divisor. Then $H(R)$ is a local ring
   with $J(H(R))=\{x_1+x_2i+x_3j+x_4k:x_1+x_2+x_3+x_4 \in J(R)\}$.
\end{lemma}
\begin{proof}
   Since $2$ is a zero-divisor in $R$, we have $2 \in J(R)$. Let us firstly prove that $x=x_1+x_2i+x_3j+x_4k$ is a zero-divisor in $H(R)$ if and only if $x_1^2+x_2^2+x_3^2+x_4^2 \in J(R)$.
   
Let $x=x_1+x_2i+x_3j+x_4k$ be a zero-divisor in $ H(R) $. Then there exists $y=y_1+y_2i+y_3j+y_4k$ in $ H(R)  $ such that $ xy=0 $ or $ yx=0 $. If $ xy=0 $, then $ \overline{x}xy\overline{y}=0 $, where $ \overline{x}=x_1-x_2i-x_3j-x_4k $ and $ \overline{y}=y_1-y_2i-y_3j-y_4k $. So, $ (x_1^2+x_2^2+x_3^2+x_4^2)(y_1^2+y_2^2+y_3^2+y_4^2)=0 $. Therefore $ x_1^2+x_2^2+x_3^2+x_4^2 $ is a zero-divisor in $ R $. Since $ R $ is local, then $ x_1^2+x_2^2+x_3^2+x_4^2 \in J(R) $. The case $ yx=0 $ can be done in a similar way.

Now, let $ x_1^2+x_2^2+x_3^2+x_4^2 \in J(R) $. Then $ x_1^2+x_2^2+x_3^2+x_4^2 $ is a zero-divisor in $ R $. So, there exists $ y \in R $ such that $ (x_1^2+x_2^2+x_3^2+x_4^2)y=0 $. Therefore $ (x_1+x_2i+x_3j+x_4k)(x_1-x_2i-x_3j-x_4k)y=0 $, which implies that either $ x_1+x_2i+x_3j+x_4k $ is a zero-divisor in $ H(R) $ or $\overline{x}y=0$, but then also $x\overline{y}=0$, so $x$ is indeed a zero-divisor in $H(R)$.
   
   Now, choose two zero-divisors $x=x_1+x_2i+x_3j+x_4k, y=y_1+y_2i+y_3j+y_4k$ in $H(R)$.
   Then $x_1^2+x_2^2+x_3^2+x_4^2, y_1^2+y_2^2+y_3^2+y_4^2 \in J(R)$, so 
   $(x_1+y_1)^2+(x_2+y_2)^2+(x_3+y_3)^2+(x_4+y_4)^2=x_1^2+x_2^2+x_3^2+x_4^2+y_1^2+y_2^2+y_3^2+y_4^2-2(x_1y_1 + x_2y_2 + x_3y_3 +x_4y_4) \in J(R)$ since $R$ is commutative. Therefore $x+y$ is a zero-divisor in $H(R)$. We have therefore proved that the set of zero-divisors in $H(R)$ is closed under addition, which is equivalent to the fact that $H(R)$ is a local ring. 
   This further implies that $J(H(R))=\{x_1+x_2i+x_3j+x_4k:x_1^2+x_2^2+x_3^2+x_4^2 \in J(R)\}$.
   Since $R$ is commutative, so $x_1^2+x_2^2+x_3^2+x_4^2=(x_1+x_2+x_3+x_4)^2 - 2(x_1x_2+x_1x_3+x_1x_4+x_2x_3+x_2x_4+x_3x_4)$. Since $2 \in J(R)$, this implies that $x_1^2+x_2^2+x_3^2+x_4^2 \in J(R)$ if and only if $(x_1+x_2+x_3+x_4)^2 \in J(R)$. Furthermore, $J(R)$ is a maximal ideal, and hence a prime ideal  in $R$. So $(x_1+x_2+x_3+x_4)^2 \in J(R)$ if and only if $x_1+x_2+x_3+x_4 \in J(R)$.
\end{proof}

We shall need the following lemmas.

\begin{lemma}
\label{sumofbinomial}
  Let $k$ be an integer. Then $\sum_{i=0}^{\lfloor k/2 \rfloor}{k \choose 2i}=\sum_{i=0}^{\lfloor (k-1)/2 \rfloor}{k \choose 2i+1}=2^{k-1}$.
\end{lemma}
\begin{proof}
  We get this easily by either adding or subtracting the summands in $2^k=(1+1)^k=\sum_{i=0}^{k}{k \choose i}$ and $0=(1-1)^k=\sum_{i=0}^{k}(-1)^i{k \choose i}$. 
\end{proof}

\begin{lemma}
\label{sumsinfields}
  Let $k \geq 2$ be an integer and $F=GF(2^r)$ for some integer $r$. Then for any $c \in F$ we have
   \begin{equation*}
      \varphi_{k}(F, c)=\begin{cases}
\frac{(-1)^k}{2^r} \left(2^{r+k-1} -2^k + (2-2^{r})^k \right), \text { if } c = 0 \text { or } c = 1, \\
\frac{(-1)^k}{2^r} \left(-2^k + (2-2^{r})^k \right), \text { otherwise}.
\end{cases}
   \end{equation*}
\end{lemma}
\begin{proof}
   Theorem 1 from \cite{Mig18} shows that 
   $\varphi_{k}(F, c)=\frac{(-1)^k}{2^r} \left( 2^r \sum\limits_{\substack{j=0 \\ c=\eta(j)}}^{k}{k \choose j} -2^k + (2-2^{r})^k \right)$, where $\eta: \ZZ \rightarrow F$ is defined by $\eta(j)=j \cdot 1$. Therefore, if $c=0$ or $c=1$, we have by Lemma \ref{sumofbinomial} that $\varphi_{k}(F, c)=\frac{(-1)^k}{2^r} \left(2^{r+k-1} -2^k + (2-2^{r})^k \right)$. For all other $c \in F$, we have 
   $\varphi_{k}(F, c)=\frac{(-1)^k}{2^r} \left(-2^k + (2-2^{r})^k \right)$.
\end{proof}

Now, we have the following theorem, which is the main result of this section.

\begin{theorem}
\label{main}
     Let $k \geq 2$ be an integer and let $R$ be a finite commutative local ring with $2^{nr}$ elements for some integers $n, r$ such that $R/J(R)$ is a field with $2^r$ elements.
   Then for any $c=x_1+x_2i+x_3j+x_4k \in H(R)$ we have
   \begin{equation*}
   \begin{split}
      \varphi_{k}(H(R), c)=\begin{cases}
(-1)^k2^{(4nk-4n-k)r} \left(2^{r+k-1} -2^k + (2-2^{r})^k \right), \text { if } \\ \hspace{5cm} x_1+x_2+x_3+x_4 \in \{0,1\} \mod J(R), \\
(-1)^k2^{(4nk-4n-k)r} \left(-2^k + (2-2^{r})^k \right), \text { otherwise}.
\end{cases}
\end{split}
   \end{equation*}
\end{theorem}
\begin{proof}
    We know that $H(R)$ is a local ring by Lemma \ref{case2local}. Furthermore, $J(H(R))=\{x_1+x_2i+x_3j+x_4k;x_1+x_2+x_3+x_4 \in J(R)\}$, so $|J(H(R))|=|R|^3|J(R)|$. Since $|J(R)|=2^{(n-1)r}$, we have $|J(H(R))|=2^{3nr+(n-1)r}=2^{(4n-1)r}$.
    Note that $F=H(R)/J(H(R))$ is a finite field isomorphic to $GF(2^r)$.
    By Theorem \ref{Th2}, we then get
    $\varphi_{k}(H(R), c)=|J(H(R))|^{k-1}\varphi_{k}(F, \overline{c})=2^{(4n-1)(k-1)r}\varphi_{k}(F, \overline{c})$, where $\overline{c}=c+J(H(R))$. 
    Observe that $\overline{c}=0$ if and only if $c \in J(H(R))$, which is equivalent to the fact that $x_1+x_2 +x_3+x_4 \in J(R)$ by the above.
    Similarly, $\overline{c}=1$ if and only if $1+x_1+x_2 +x_3+x_4 \in J(R)$.
    Lemma \ref{sumsinfields} now proves the assertion.
\end{proof}

\bigskip

\section{\textbf{Rings of odd order}}

\noindent
In this section, we examine the case when $R$ is a finite commutative local ring of an odd order. It turns out that the situation in this case is somewhat more complicated. 

%Let $F$ be a field of odd order. The following two propositions describe the isomorphism between $H(F)$ for a finite field $F$ and the ring of 2-by-2 matrices over $F$.
%\begin{proposition}\cite[Proposition 2.5.3]{Dav}
%Let $q$ be an odd prime power. There exists $x, y \in F_q $,
%such that $x^2 + y^2 + 1 = 0$.
%\end{proposition}
    
%\begin{proposition}\cite[Proposition 2.5.2]{Dav}
%\label{iso}
%Let $F$ be a field, not of characteristic $2$. Assume that there exist $x, y \in F$, %such that $x^2+y^2+1=0$. Then  $\psi: M_2(F) \rightarrow H(F)$ defined by
%$$
%\psi(a+bi+cj+dk)=\left(\begin{array}{cc}
%a+bx+dy & -by+c+dy\\
%-by-c+dy & a-bx-dy
%\end{array}\right).
%$$
%is a ring isomorphism.
%\end{proposition}

\begin{comment}
\begin{proof}
One checks that $\psi\left(A_1 A_2\right)=\psi\left(A_1\right) \psi\left(A_2\right)$ for $A_1, A_2 \in M_{2}(K)$. Since $\psi$ is an $F$-linear map between two $F$-vector spaces of the same dimension 4 , to prove that $\psi$ is an isomorphism it is enough to show that $\psi$ is injective. But $\psi\left(\begin{array}{cc}
a & b\\
c & d
\end{array}\right)=0$ leads to a 4-by-4 homogeneous linear system in the variables $a, b, c, d$ with determinant
$$
\frac{1}{2}\left|\begin{array}{cccc}
1 & 0 & 0 & 1 \\
-x & y & y & x \\
0 & 1 & -1 & x \\
-y & -x & -x & y
\end{array}\right|=-\frac{x^2+y^2}{4}=\frac{1}{4}.
$$
Therefore, $ \left(\begin{array}{cc}
a & b\\
c & d
\end{array}\right)=0 $.
\end{proof}
\end{comment}

In view of Theorem \ref{Th1}, we have to investigate the problem in $M_2(F)$ for a finite field $F$ of an odd order.
We can immediately establish the following bound.

\begin{lemma}
\label{bound}
Let $ F $ be a finite field of an odd order $q$. If $q \geq 9$ then
$\varphi_{2}(M_{2}(F), C) \geq q^4 - 8 q^3$ for any $C \in M_2(F)$. 
\end{lemma}
\begin{proof}
    Choose $C \in M_2(F)$ and $x, y, z \in F$ and define $A=\begin{pmatrix}
x & y\\
z & 0
\end{pmatrix} \in M_{2}(F)$. Then $(A+\lambda I)+(C-A-\lambda I)=C$.
Denote $\Spec(A)=\{a_1,a_2\}$ and $\Spec(C-A)=\{b_1,b_2\}$. (Note that $\Spec(A)$ and $\Spec(C-A)$ are in general subsets of some field extension $K$ over $F$.) Denote $X=(-\Spec(A)) \cup (-\Spec(A)+1) \cup (\Spec(C-A)) \cup (\Spec(C-A)+1)$ and observe that $|X \cap F| \leq 8$. If $\lambda \notin X$, then $(A+\lambda I)+(C-A-\lambda I)=C$ is a sum of two exceptional units. Since $(A+\lambda I)_{22}=\lambda$, we see that we get  different sums for every $\lambda \in F \setminus X$ and every $x,y,z \in F$, so 
$\varphi_{2}(M_{2}(F), C) \geq q^3(q - 8)$.
\end{proof}

\begin{lemma}
\label{sumstomatrices}
 Let $ R $ be a ring and $ c \in R $. Then
\begin{equation*}
\varphi_{2}(R, c) = | \{(x_{1}, x_{2}) \in (R^{**})^{2} : x_{1} + x_{2} = c \}|=| \{ x \in R : x, x-1, x-c, x+1-c \in R^{*} \}|.
\end{equation*}
\end{lemma}
\begin{proof}
    Observe that $x_2=c-x_1$ is an exceptional unit, so $c-x_1$ and $1-(c-x_1)$ are in $R^*$.
\end{proof}

In the following lemmas, we will now try to calculate the number of decompositions into sums of two exceptional units for different types of two by two matrices.

\begin{lemma}
\label{phi2ofzero}
Let $ F $ be a finite field of an odd order $q$. Then
\begin{equation*}
\varphi_{2}(M_{2}(F), 0)= q^4 - 3 q^3 + 6 q.
\end{equation*}
\end{lemma}
\begin{proof}
By Lemma \ref{sumstomatrices} we need to count the number of matrices $A$ in $M_{2}(F) $ such that $ A $, $ A+I $ and $ A-I $ are invertible.
\begin{comment}
Suppose therefore that $ A=\begin{pmatrix}
a & b\\
c & d\\
\end{pmatrix} \in M_{2}(F) $ and $ A $, $ A+I $ and $ A-I $ are invertible. 
Then $ ad-bc$, $ (a+1)(d+1)-bc $ and $ (a-1)(d-1)-bc $ are nonzero.

For finding the order of $ \varphi_{2}(M_{2}(F), 0_{M_{2}(F)}) $, we first find the order of the following set:
\begin{equation*}
\mathcal{S}=\{ \begin{pmatrix}
a & b\\
c & d\\
\end{pmatrix} \in M_{2}(F) : ad=bc \quad \text{or} \quad (a+1)(d+1)=bc \quad \text{or} \quad (a-1)(d-1)=bc \}.
\end{equation*}

First, we deal with case $ ad=bc $. Therefore, we have the following two cases:

(i) $ ad=bc=0 $ or (ii) $ ad=bc\neq 0 $.

In (i), the number of matrices which $ ad=0 $ is equal to $ 2q-1 $. Similarily, the number of matrices which $ ad=0 $ is $ 2p-1 $. So, the number matrices in (i) is $ (2q-1)^{2} $.

For finding the number of matrices in (ii), since $ d $ is invertible, then $ a=d^{-1}bc \neq 0 $. So, the number of matrices in (ii) is $ (|F|-1)^{3} $. Therefore, if $ ad=bc $, there are $ ((|F|-1)^{3})+(2|F|-1)^{2} $ matrices. Similarily, the case $ (a+1)(d+1)=bc $, as well as case $ (a-1)(d-1)=bc $, has $ ((|F|-1)^{3})+(2|F|-1)^{2} $ matrices.

Now, we deal with the intesctions of cases. If $ ad=bc $ and $ (a+1)(d+1)=bc $, then $ a=-d-1 $. Therefore, there are $ p^{•} $
\end{comment}

Recall that there exists an invertible matrix $ B \in M_{2}(F) $ such that $ A=PBP^{-1} $ for some $ P \in M_{2}(F) $, where either $ B=\begin{pmatrix}
0 & b\\
1 & d\\
\end{pmatrix} $
($ b $, $ d+1-b  $, and $ d-1+b $ are nonzero), or 
$ B=\begin{pmatrix}
a & 0\\
0 & a\\
\end{pmatrix} $
($ a, a+1 $ and $ a-1$ are nonzero). 
In the first case, we have $ q^{2}-3q+3 $ matrices, and in the second case, we have 
$ q-3 $
%$ (q-3)^{2} $ 
matrices. Noting that $ P_{1}BP_{1}^{-1}=P_{2}BP_{2}^{-1} $ if and only if $ P_{2}^{-1}P_{1} \in C_{GL_{2}(F)}(B) $, we need to compute $ |C_{GL_{2}(F)}(B)| $, since then the number of all matrices with the rational form equal to $B$ is equal to $\frac{|GL_{2}(F)|}{|C_{GL_{2}(F)}(B)|}$.
Obviously in the second case, $B$ is central, so $ C_{GL_{2}(F)}(B)= GL_{2}(F)$ and there exist exactly $q-3$ matrices $A$ with its rational form being a scalar matrix and $A, A+I, A-I$ being invertible.

%\begin{equation*}
%\varphi_{2}(M_{2}(R), 0_{M_{2}(R)})=|GL_{2}(F)|(\frac{|F|^{2}-3|F|+3}{|C_{GL_{2}(F)}(B_{1})|}+\frac{(|F|-3)^{2}}{|C_{GL_{2}(F)}(B_{1})|}).
%\end{equation*}
Let us therefore examine the first case. Let $ C=\begin{pmatrix}
x & y\\
z & w\\
\end{pmatrix} $ be in $ C_{GL_{2}(F)}(B) $. Then $ C=\begin{pmatrix}
x & bz\\
z & x+dz\\
\end{pmatrix} $, where $ x^{2}+xdz-bz^{2} $ is nonzero. 
So, let us examine when $x^{2}+xdz-bz^{2}=0$. Obviously, if $x=0$ then $z=0$ and vice versa. So, assume that $z \neq 0$ and denote $t=xz^{-1}$. Then 
\begin{equation}
    \label{eq:1}
    t^2+dt-b=0.
\end{equation} Consider the mapping $f: F \rightarrow F$, defined by $f(x)=x^2$. If $f(x)=f(y)$ then $x=y$ or $x=-y$, so the fact that $F$ is of an odd order now yields $|\im(f)|=\frac{q-1}{2}+1=\frac{q+1}{2}$. Now, consider the following three cases. 

{\bf Case 1}: Suppose that $d^2+4b \notin \im(f)$. Then Equation (\ref{eq:1}) has no solutions, so we get $|C_{GL_{2}(F)}(B)| = q^2-1$ in this instance. Using the well known fact that $|GL_{2}(F)|=(q^2-1)(q^2-q)$, this implies that we have $\frac{|GL_{2}(F)|}{q^2 -1}=q^2-q$ different matrices for each choice of $b$ and $d$ such that $d^2+4b \notin \im(f)$.
So, how many choices of $b$ and $d$ are such that $d^2+4b \notin \im(f)$? We can choose any $d \in F$ (we have $q$ possibilities). Now consider the mapping $g_d: F \rightarrow F$ defined by $g_d(b)=d^2+4b$. Since $g_d$ is injective and $|\im(f)|=\frac{q+1}{2}$, exactly $q-\frac{q+1}{2}=\frac{q-1}{2}$ choices of $b$ will give $g_d(b) \notin \im(f)$. However, we also have further constraints that $b \neq 0$, $b \neq d+1$ and $b \neq 1-d$. However, $b=0$ gives us $g_d(b)=d^2 \in \im(f)$. Similarly, $b=d+1$ yields $g_d(b)=(d+2)^2 \in \im(f)$ and $b=1-d$ yields $g_d(b)=(d-2)^2 \in \im(f)$, so these three possibilities do not appear here. This implies that we have $q\frac{q-1}{2}$ choices for $b$ and $d$ such that $d^2+4b \notin \im(f)$. This yields altogether $(q^3-q^2)(\frac{q-1}{2})$ matrices.

{\bf Case 2}: Suppose that $d^2+4b=0$. Then Equation (\ref{eq:1}) has a solution $t=-2^{-1}d$. Note that $b \neq 0$ implies $d \neq 0$, so for any of the $q$ choices for $x$, we have $q-1$ choices for $z$, therefore $|C_{GL_{2}(F)}(B)| = q^2-q$, yielding $\frac{|GL_{2}(F)|}{q^2 -q}=q^2-1$ matrices. Observe that $b=d+1$ implies that $d^2+4b=(d+2)^2=0$, so $d=-2$ and $b=-1$. Similarly, $b=1-d$ implies that $d^2+4b=(d-2)^2=0$, so $d=2$ and $b=-1$. Note also that $2 \neq -2$, since $F$ is a of odd order. Therefore, $d \in F \setminus \{-2,0,2\}$ can be chosen arbitrarily and then $b=-4^{-1}d^2$, giving us $q-3$ possibilities for choosing $b$ and $d$. In this case, we thus have $(q-3)(q^2-1)$ matrices.

{\bf Case 3}: Suppose finally that $0 \neq d^2+4b \in \im(f)$, say $d^2+4b=\alpha^2$ for some $\alpha \in F^*$. Then Equation (\ref{eq:1}) has a solution $t=2^{-1}(-d \pm \alpha)$, so for any choice of non-zero $x$, we have $q-2$ choices for $z$, giving us $(q-1)+(q-1)(q-2)=(q-1)^2$ elements of $C_{GL_{2}(F)}(B)$, thus yielding $q(q+1)$ matrices in this case.  Similar reasoning as in Case 1 shows that for any $d \in F$ exactly $\frac{q-1}{2}$ choices of $b$ will give $g_d(b)=d^2+4b \in \im(f) \setminus \{0\}$. 
We also have three additional constraints: $b \neq 0, b \neq d+1$ and $b \neq 1-d$.

Suppose firstly that $\charac{F} \neq 3$.
If $d=0$ therefore $b \neq 1$ ($b \neq 0$ is automatically true since $g_d(b) \neq 0$). If $d=\pm 1$ we have $b \neq 0,2$. If $d= \pm 2$ then $b \neq 0,3$ (since $b=-1$ gives $g_d(b)=0$ so this case does not appear here).  For all other choices of $d$, the elements $d+1, 1-d$ and $0$ are distinct (and $g_d(b) \in \im(f) \setminus \{0\}$ when $b$ equals any of them), so we have $\frac{q-1}{2}-3$ choices for $b$. This means that have exactly $(\frac{q-1}{2}-1)+4(\frac{q-1}{2}-2)+(q-5)(\frac{q-1}{2}-3)=\frac{q^2- 7q + 12}{2}$ choices for $d$ and $b$.

If $\charac{F}=3$, then we have the following constraints: if $d=0$ we have $b \neq 1$ ($b \neq 0$ is automatically true since $g_d(b) \neq 0$). If $d=\pm 1$ we have $b \neq 0$ ($b=2$ gives $g_d(b)=0$, so it does not appear here). For all other choices of $d$, the elements $d+1, 1-d$ and $0$ are distinct (and $g_d(b) \in \im(f) \setminus \{0\}$ when $b$ equals any of them), so we have $\frac{q-1}{2}-3$ choices for $b$. This means that have exactly $3(\frac{q-1}{2}-1)+(q-3)(\frac{q-1}{2}-3)=\frac{q^2- 7q + 12}{2}$ choices for $d$ and $b$.

Thus, in Case 3, we have $\frac{q^2-7q+12}{2}q(q+1)$ matrices.

%Therefore, $ |F|^{2} -2|F| \leq |C_{GL_{2}(F)}(B)| \leq  |F|^{2} -2|F| $.  

So, if we sum it all together, we have $(q-3)+(q^3-q^2)(\frac{q-1}{2})+(q-3)(q^2-1)+\frac{q^2-7q+12}{2}q(q+1)=q^4 - 3 q^3 + 6 q$
matrices $A \in M_2(F)$ such that $A, A-I$ and $A+I$ are invertible and thus the result is proven.
%let again $ C=\begin{pmatrix}
%x & y\\
%z & w\\
%\end{pmatrix} $ be in $ C_{GL_{2}(F)}(B) $. If $ a=d $, then $ B \in Z(GL_{2}(F)) $. If $ a \neq d $, then $ C=\begin{pmatrix}
%x & 0\\
%0 & w\\
%\end{pmatrix} $. So, $ |C_{GL_{2}(F)}(B)|=(|F|-1)^{2} $.

%Consequently, 
%\begin{equation*}
%\frac{|GL_{2}(F)|}{|F|^{2} -1}+\frac{|GL_{2}(F)|}{(|F|-1)^{2}}+1 \leq \varphi_{2}(M_{2}(F), 0) \leq \frac{|GL_{2}(F)|}{|F|^{2} -2|F|}+\frac{|GL_{2}(F)|}{(|F|-1)^{2}}+1.
%\end{equation*}
%matrices in this case.
\end{proof}
%%%%%%%%%%%%%%%%%%%%%%%%%%%%%%%%
%%%%%%%%%%%%%%%%%%%%%%%%%%%%%%%%%
%%%%%%%%%%%%%%%%%%%%%%%%%%%%%%%%%%%

%%%%%%%%%%%%%%%%%%%%%%%%%%%%%%%%%
%%%%%%%%%%%%%%%%%%%%%%%%%%%%%%%
%%%%%%%%%%%%%%%%%%%%%%%%%%%%%
%%%%%%%%%%%%%%%%%%%%%%%%%%%%%%%%%
%%%%%%%%%%%%%%%%%%%%%%%%%%%%%%%
%%%%%%%%%%%%%%%%%%%%%%%%%%%%%

\begin{lemma}
\label{phi2ofinvertible}
Let $ F $ be a finite field of an odd order $q$ and let $C \in M_2(F)$ be an invertible matrix. Then
\begin{equation*}
\varphi_{2}(M_{2}(F), C) \leq q^4 - 2 q^3 - q^2  + 3 q.
\end{equation*}
Furthermore, if $C=1$, the above inequality becomes an equality.
\end{lemma}
\begin{proof}
Note that $A+B=C$ for exceptional units $A$ and $B$ implies that $AC^{-1}+BC^{-1}=1$. Obviously, $AC^{-1}$ and $BC^{-1}$ are invertible matrices. Since $B=A-C$ is invertible, we see that $(C-A)C^{-1}=1-AC^{-1}$ is invertible as well, so $AC^{-1}$ is an exceptional unit, and of course we can reason similarly for $BC^{-1}$. Since the mapping $X \mapsto XC^{-1}$ is a bijection on the set of all invertible matrices, we conclude that $\varphi_{2}(M_{2}(F), C) \leq \varphi_{2}(M_{2}(F), 1)$.
By Lemma \ref{sumstomatrices}, we know that to find $\varphi_{2}(M_{2}(F), 1)$ we need to count the number of matrices $A$ in $M_{2}(F) $ such that $ A $ and $ A-I $ are invertible, i.e. matrices without $0$ and $1$ as their eigenvalues. The result now follows from \cite[Corollary 2.1]{morrison}.
\end{proof}

\begin{lemma}
\label{phi2ofidempotent}
Let $ F $ be a finite field of an odd order $q$ and let $C \in M_2(F)$ be an idempotent matrix of rank one. Then
\begin{equation*}
\varphi_{2}(M_{2}(F), C)= q^4 - 4 q^3 + 5 q^2 - 4 q + 4.
\end{equation*}
\end{lemma}
\begin{proof}
If $P \in M_2(F)$ is invertible then $A \in M_2(F)$ is an exceptional unit if and only if $PAP^{-1}$ is an exceptional unit. Therefore, $\varphi_{2}(M_{2}(F), C)=\varphi_{2}(M_{2}(F), PCP^{-1})$, so we can assume that
$ C=\begin{pmatrix}
1 & 0\\
0 & 0\\
\end{pmatrix}$.
Let $ A=\begin{pmatrix}
a & b\\
c & d\\
\end{pmatrix}$. Then $C=A+(C-A)$ is a sum of two exceptional units if and only if the following conditions hold:
\begin{equation}
 \label{conditions}
 \begin{split}
  ad-bc \neq 0,  \, \,
  (1-a)(1-d)-bc \neq 0,  \\
  (1-a)d+bc \neq 0 \, \text { and } \,
  a(1+d) - bc \neq 0.
  \end{split}
\end{equation}
Denote $t=bc \in F$. So conditions (\ref{conditions}) are equivalent to the fact that
\[t \notin \{ad, (1-a)(1-d), (a-1)d, a(1+d)\}.\]
Now, consider the following cases:
\begin{enumerate}
    \item 
    $a=0; d=0$ or $a=1; d = 0$: then $t \notin \{0, 1\}$, so $b$ can be any nonzero element in $F$ and $c$ can be any nonzero element in $F$ apart from $tb^{-1}$, which gives us $2(q-1)(q-2)$ choices for $A$;

    \item 
    $a=0; d=1$ or $a=1; d = -1$: then $t \notin \{0, -1\}$, so again $b$ can be any nonzero element in $F$ and $c$ can be any nonzero element in $F$ apart from $-tb^{-1}$, which gives us another $2(q-1)(q-2)$ matrices $A$;

    \item 
    $a=0; d \neq 0,1$: then $t \notin \{0, -d, 1-d\}$, so $b$ can be any nonzero element in $F$ and $c$ can be any nonzero element in $F$ apart from $-db^{-1}$ and $-(1+d)b^{-1}$, which gives us $(q-1)(q-2)(q-3)$ choices for $A$;

    \item 
    $a=1; d \neq 0, -1$: then $t \notin \{0, d, 1+d\}$, so $b$ can be any nonzero element in $F$ and we have $q-3$ choices for $c \in F$, which gives us $(q-1)(q-2)(q-3)$ choices for $A$;

    \item 
    $d=0; a = 2^{-1}$: then $1-a=a$, so $t \notin \{0, a\}$, yielding $(q-1)(q-2)$ choices for $A$;

    \item 
    $d=0; a \neq 0, 1, 2^{-1}$: then $1-a \neq a$, so $t \notin \{0, a, 1-a\}$, therefore we have $q-3$ choices for $a$, yielding together $(q-1)(q-3)^2$ choices for $A$;

    \item 
    $d = 1; a =-1$ or $d=-1; a=2$: then $t \notin \{0, -1, -2\}$, therefore we have $2(q-1)(q-3)$ choices for $A$;

   \item 
    $d \neq 0, 1, -1; a = -d$: then $t \notin \{-d^2, 1-d^2, -d^2-d\}$, and note that this set does not contain $0$. Therefore, we have $2q-1$ choices for $b$ and $c$ such that $t=0$ and $q-1$ choices for $b$ and $q-4$ choices for $c$ such that $t \neq 0$. Since we have $q-3$ choices for $d$, these together give us $(q-3)(2q-1+(q-1)(q-4))$ choices for $A$;

   \item 
    $d \neq 0, 1, -1; a = 1-d$: then $t \notin \{-d^2, 1-d^2, d-d^2\}$. Again, this set does not contain $0$, so we have $q-3$ choices for $d$ and  $2q-1+(q-1)(q-4)$ choices for $b$ and $c$, yielding together $(q-3)(2q-1+(q-1)(q-4))$ choices for $A$;

   \item 
    $d \neq 0, 1, -1; 2a = 1-d$: then $t \notin \{a-2a^2, 2a-2a^2, -1+3a-2a^2 \}$. Since $2a^2-3a+1=(a-1)(2a-1)$, we see that this set does not contain zero. Therefore we have $q-3$ choices for $d$, and $2q-1 + (q-1)(q-4)$ choices for $b$ and $c$, yielding together $(q-3)(2q-1+(q-1)(q-4))$ choices for $A$;

    \item
    $d=1; a \neq 0, 1, -1$: then $t \notin \{0, a, a-1, 2a \}$ therefore we have $q-3$ choices for $a$, $q-1$ choices for $b$ and $q-4$ choices for $c$, yielding together $(q-1)(q-3)(q-4)$ choices for $A$;

    \item
    $d=-1; a \neq 0, 1, 2$: then $t \notin \{0, -a, 1-a, 2-2a \}$ therefore we have $q-3$ choices for $a$, $q-1$ choices for $b$ and $q-4$ choices for $c$, yielding together $(q-1)(q-3)(q-4)$ choices for $A$; and finally

    \item
    $d \neq 0,1,-1; a \neq 0, 1, 1-d,-d,2^{-1}(1-d)$: then $t \notin \{ad, (1-a)(1-d), (a-1)d, a(1+d)\}$ and this set has exactly $4$ nonzero elements. Therefore, we have $q-3$ choices for $d$, $q-5$ choices for $a$, $2q-1+(q-1)(q-5)$ choices for $b$ and $c$, yielding together $(q-3)(q-5)(2q-1+(q-1)(q-5))$ choices for $A$.
\end{enumerate}
One can easily verify that all cases are mutually exclusive and all together they cover every possibility. This gives us $q^4 - 4 q^3 + 5 q^2 - 4 q + 4$ different matrices $A$, thus $q^4 - 4 q^3 + 5 q^2 - 4 q + 4$ different sums of two exceptional units and the lemma is proven.
\end{proof}

%\begin{comment}

\begin{lemma}
\label{scalartimesidem}
Let $F$ be a field of order $q \geq 9$ and $\lambda \in F^*$. Then 
$$\frac{q-8}{q}\varphi_{2}(M_{2}(F), C) \leq \varphi_{2}(M_{2}(F), \lambda C) \leq \frac{q}{q-8}\varphi_{2}(M_{2}(F), C).$$
\end{lemma}
\begin{proof}
  Suppose $A+B=C$ for some exceptional units $A,B \in M_2(F)$. 
  Then $(\lambda A - \mu I)+ (\lambda B + \mu I) = \lambda C$ for any $\mu \in F$. Denote $\Spec(A)=\{a_1,a_2\}$ and $\Spec(B)=\{b_1,b_2\}$ (note that we may of course have $a_1=a_2$ or $b_1=b_2$). Now, $\Spec(\lambda A - \mu I)= \lambda \Spec(A)- \mu$, so $\lambda A - \mu I$ is an exceptional unit in $M_2(F)$ if and only if $\mu \notin \{\lambda a_1, \lambda a_2, \lambda a_1 - 1, \lambda a_2 -1\}$. Similarly, $\lambda B + \mu I$ is an exceptional unit in $M_2(F)$ if and only if $\mu \notin \{-\lambda b_1, -\lambda b_2, 1- \lambda b_1, 1 - \lambda b_2\}$. Since $q \geq 9$, there exist at least $q-8$ different elements $\mu \in F$ such that $\mu \notin \{\lambda a_1, \lambda a_2, \lambda a_1 - 1, \lambda a_2 -1\} \cup \{-\lambda b_1, -\lambda b_2, 1- \lambda b_1, 1 - \lambda b_2\}$. This gives us $q-8$ different decompositions of $\lambda C$ into a sum of two exceptional units.

  Now, let $C=A'+B'$ be another decomposition of $C$ into a sum of two exceptional units $A',B' \in M_2(F)$. Suppose that the above construction yields the decomposition of  $\lambda C = (\lambda A' - \mu' I) + (\lambda B' + \mu' I)$ for some $\mu' \in F$ and assume that this is the same decomposition that we obtained from the decomposition $C=A+B$, i.e., $\lambda A - \mu I = \lambda A' - \mu' I$ and $\lambda B + \mu I = \lambda B' + \mu' I$. This of course happens only if $A=A' + \tau I$ and $B=B' - \tau B$ for some $\tau \in F$.  

  Consequently, $\varphi_{2}(M_{2}(F), \lambda C) \geq \frac{q-8}{q}\varphi_{2}(M_{2}(F), C)$. Since $\lambda$ is invertible, we have $\varphi_{2}(M_{2}(F), C) \geq \frac{q-8}{q}\varphi_{2}(M_{2}(F), \lambda C)$, so 
 \[\frac{q-8}{q}\varphi_{2}(M_{2}(F), C) \leq \varphi_{2}(M_{2}(F), \lambda C) \leq \frac{q}{q-8}\varphi_{2}(M_{2}(F), C).\]
\end{proof}
%\end{comment}

We now have the following corollary.

\begin{corollary}
 \label{phi2ofrankone}
Let $ F $ be a finite field of an odd order $q \geq 9$ and let $C \in M_2(F)$ be a non-nilpotent matrix of rank one. Then
\begin{equation*}
\frac{q-8}{q}(q^4 - 4 q^3 + 5 q^2 - 4 q + 4) \leq \varphi_{2}(M_{2}(F), C) \leq \frac{q}{q-8}(q^4 - 4 q^3 + 5 q^2 - 4 q + 4).
\end{equation*}
\end{corollary}
\begin{proof}
Since $C$ is a non-nilpotent matrix of rank one, it has two eigenvalues: $0$ and some $\lambda \in F^*$. Therefore, we can assume that $C$ is similar to matrix $\begin{pmatrix}
\lambda & 0\\
0 & 0\\
\end{pmatrix}=\lambda \begin{pmatrix}
1 & 0\\
0 & 0\\
\end{pmatrix}$. The result now follows from Lemmas \ref{phi2ofidempotent} and \ref{scalartimesidem}.
\end{proof}

\begin{lemma}
\label{phi2ofnilpotent}
Let $ F $ be a finite field of an odd order $q$ and let $C \in M_2(F)$ be a nilpotent matrix of rank one. Then
\begin{equation*}
\varphi_{2}(M_{2}(F), C)= q^4 - 4 q^3 + 3 q^2 + 2 q.
\end{equation*}
\end{lemma}
\begin{proof}
By choosing a suitable basis, we can assume that
$ C=\begin{pmatrix}
0 & 1\\
0 & 0\\
\end{pmatrix}$.
Let $ A=\begin{pmatrix}
a & b\\
c & d\\
\end{pmatrix}$. Then $C=A+(C-A)$ is a sum of two exceptional units if and only if the following conditions hold:
\begin{equation}
 \label{conditions1}
 \begin{split}
  bc \neq ad,  \, \,
  bc \neq (1-a)(1-d),  \\
  (b-1)c \neq ad \, \text { and } \,
  (b-1)c \neq (1+a)(1+d).
  \end{split}
\end{equation}

Assume first that $c=0$. Then $b \in F$ and $a, d \in F \setminus \{-1,0,1\}$ all satisfy conditions (\ref{conditions1}), so we have $q(q-3)^2$ choices for $A$ in this case. 

Assume now that $c \neq 0$. Then conditions (\ref{conditions1}) are equivalent to the fact that
\[b \notin \{adc^{-1}, (1-a)(1-d)c^{-1}, adc^{-1}+1, (1+a)(1+d)c^{-1}+1\}.\]
Now, consider the following cases:
\begin{enumerate}
    \item 
    $a=1-d$: then $b \notin X=\{(1-d)dc^{-1},  (1-d)dc^{-1}+1, (2-d)(1+d)c^{-1}+1\}$. If $c=2$ then $|X|=2$, so we have $q(q-2)$ choices for $a,b,c$ and $d$. If $c \neq 2$ then $|X|=3$ and we have
    $q(q-2)(q-3)$ choices for $A$.
    \item 
    $a=-1-d$: then $b \notin X=\{(-d)(1+d)c^{-1},  (2+d)(1-d)c^{-1}, (-d)(1+d)c^{-1}+1\}$. If $c=-2$ then $|X|=2$, so we have $q(q-2)$ choices for $a,b,c$ and $d$. If $c \neq -2$ then $|X|=3$ and we have
    $q(q-2)(q-3)$ choices for $A$.
    \item 
    $c \neq 2$ and $a=1-c-d$: then $b \notin X=\{(1-c-d)dc^{-1},(c+d)(1-d)c^{-1}, (2-c-d)(1+d)c^{-1}+1\}$. Since $c \neq 2$, we have $|X|=3$, so we have
    $q(q-2)(q-3)$ choices for $A$.
    \item 
    $c \neq -2$ and $a=-1-c-d$: then $b \notin X=\{(-1-c-d)dc^{-1},(2+c+d)(1-d)c^{-1}, (-1-c-d)dc^{-1}+1\}$. Since $c \neq -2$, we have $|X|=3$, so we have
    $q(q-2)(q-3)$ choices for $A$.
    \item 
    $c \neq -2,2$ and $2a+2d=-c$: then $b \notin X=\{(2^{-1}c-d)dc^{-1},(1-2^{-1}c+d)dc^{-1}(1-d)c^{-1}, 2^{-1}c-d)dc^{-1}+1\}$. We have $|X|=3$, so we have
    $q(q-3)^2$ choices for $A$. 
        
    \item 
    $a \notin \{1-d, -1-d, 1-c-d, -1-c-d, -2^{-1}c-d \}$: If $c=2$ or $c=-2$ then $|\{1-d, -1-d, 1-c-d, -1-c-d, 2^{-1}c-d \}|=3$, so there are $q-3$ choices for $a$. The fact that $c \neq -2a -2d$ implies that $b \notin X=\{adc^{-1}, (1-a)(1-d)c^{-1}, adc^{-1}+1, (1+a)(1+d)c^{-1}+1\}$ and $|X|=4$, so there are $q-4$ choices for $b$. This gives us $2q(q-3)(q-4)$ choices for $A$. 
    If $c \neq -2,2$ then $|\{1-d, -1-d, 1-c-d, -1-c-d, -2^{-1}c-d\}|=5$, so there are $q$ choices for $d$ and $q-5$ choices for $a$. We have $|X|=4$, so there are $q-4$ choices for $b$, and this gives us $q(q-3)(q-4)(q-5)$ choices for $A$.
\end{enumerate}
Again, all cases are mutually exclusive and together they cover every possibility. This gives us $q^4 - 4 q^3 + 3 q^2 + 2 q$ different sums of two exceptional units and the lemma is proven.
\end{proof}

Now, we gather the above results in the following theorem, which is the main result of this section.

\begin{theorem}
\label{main2}
    Let $p > 2$ be a prime and let $ R $ be a finite commutative local ring of order $p^{nr}$ such that $R/J(R)$ is a field with $p^r$ elements. 
    Let $c \in H(R)$ and denote $\overline{c}=c+J(H(R))$.  Then
    \begin{equation*}
    \varphi_{2}(H(R), c)= \begin{cases} p^{(4n-3)r}(p^{3r} - 3 p^{2r} + 6), \text { if } \overline{c}=0, \\
    p^{(4n-3)r}(p^{3r} - 2 p^{2r} - p^r + 3), \text { if } \overline{c}=1, \\
    p^{4(n-1)r}(p^{4r} -4 p^{3r} + 5 p^{2r} - 4 p^r + 4), \text { if } 0,1 \neq \overline{c} \text { is an idempotent}, \\
    p^{(4n-3)r}(p^{3r} - 4 p^{2r} + 3p^r + 2), \text { if } 0 \neq \overline{c} \text { is a nilpotent}.
    \end{cases}
    \end{equation*}
    If $c$ is invertible, we have $\varphi_{2}(H(R), c) \leq p^{(4n-3)r}(p^{3r} - 2 p^{2r} - p^r + 3)$. 
    Furthermore, if $p^r \geq 9$ then $\varphi_{2}(H(R), c) \geq p^{(4n-3)r}(p^{3r} - 8)$ and if additionally $\overline{c}^2=\lambda \overline{c}$ for some $0, 1 \neq \lambda \in F$, we have
    $\varphi_{2}(H(R), c) \leq \frac{p^{(4n-3)r}}{p^r-8}(p^{4r} - 4p^{3r} + 5p^{2r} - 4p^r + 4)$.
\end{theorem}
\begin{proof}
    By Theorem \ref{Th3}, we have $H(R)/J(H(R)) \cong H(R)/H(J(R)) \cong H(R/J(R)) \cong H(F)$, where $ F=R/J(R)$ is a field with $|F|=p^r$. By Corollary \ref{Th2}, we know that we have $\varphi_{2}(H(R), c)=|J(H(R))|\varphi_{2}(H(R)/J(H(R)), \overline{c})=|J(R)|^{4}\varphi_{2}(M_2(F), \psi(\overline{c}))=p^{4(n-1)r}\varphi_{2}(M_2(F), \psi(\overline{c}))$, where $\psi$ denotes the isomorphism from $H(F)$ to $M_2(F)$, which is guaranteed by Theorem \ref{Th1}. 
    
    If $\overline{c}=0$, then $\psi(\overline{c})=0$, so Lemma \ref{phi2ofzero} now yields $\varphi_{2}(H(R), c)=p^{4(n-1)r}(p^{4r} - 3 p^{3r} + 6 p^r)$. On the other hand, if $c \in H(R)$ is invertible, then $\psi(\overline{c})$ is invertible as well and Lemma \ref{phi2ofinvertible} tells us that
    $\varphi_{2}(H(R), c) \leq p^{4(n-1)r}(p^{4r} - 2 p^{3r} - p^{2r} + 3 p^r)$ and $\varphi_{2}(H(R), c) = p^{4(n-1)r}(p^{4r} - 2 p^{3r} - p^{2r} + 3 p^r)$ if $\overline{c}=1$. Assume now that $0,1 \neq \overline{c}$ is an idempotent. By Lemma \ref{phi2ofidempotent}, we have $\varphi_{2}(H(R), c)=p^{4(n-1)r}(p^{4r} -4 p^{3r} + 5 p^{2r} - 4 p^r + 4)$. Finally, if $0 \neq \overline{c}$ is nilpotent, then Lemma \ref{phi2ofnilpotent} gives the desired result. 
    
    In case $p^r \geq 9$, the lower bound follows from Lemma \ref{bound} and the upper bound follows from Corollary \ref{phi2ofrankone}.
\end{proof}

\bigskip

\begin{comment}
\begin{remark}
If $ R $ is a finite commutative ring of even order, then $ H(\frac{R}{J(R)}) \cong \prod_{i=1}^{4}\prod_{j=1}^{t}F_{i} $, while $ \frac{H(R)}{J(H(R))}\cong \prod_{j=1}^{s} M_{n_{i}}(K_{j})$, where $ F_{i} $ and $ K_{j} $ are finite fields. So, in general, for an arbitrary commutative ring $ R $, it is not always possible to compute $ \varphi_{k}(H(R), c) $.
\end{remark}
\end{comment}

\begin{comment}
\begin{theorem}
     Let $k$ be an integer and let $R$ be a finite commutative ring with $2^{n(r_{1}+\dots+r_{t})}$ elements for some integers $n, r_{j}$ such that $R/J(R) \cong \prod_{j=1}^{t}F_{j}$, where $ F_{j} $ are fields with $| F_{j}|=2^{r_{j}}$.
   Then for any $c \in H(R)$ we have
   \begin{align*}
      &\varphi_{k}(H(R), c)=\\
      &\begin{cases}
\prod_{j=1}^{t}(-1)^k2^{(4nk-4n-k)r} \left(2^{r_{j}+k-1} -2^k + (2-2^{r_{j}})^k \right), \text { if } c \in J(H(R) \cup (1 + J(H(R)), \\
\prod_{j=1}^{t}(-1)^k2^{(4nk-4n-k)r_{j}} \left(-2^k + (2-2^{r_{j}})^k \right), \text { otherwise}.
\end{cases}
   \end{align*}
\end{theorem}
\begin{proof}
Since $R/J(R) \cong \prod_{j=1}^{t}F_{j}$, then $H(R)/J(H(R))$.
    By Theorem \ref{Th2}
\end{proof}
\end{comment}
%%%%%%%%%%%%%%%%%%%%%%%%%%%%%
%%%%%%%%%%%%%%%%%%%%%%%%%%%
%%%%%%%%%%%%%%%%%%%%%%%%%%%

\end{document}